\documentclass[12pt]{amsart} 
\usepackage{amsthm, amssymb, color}
\usepackage{amsmath}
\usepackage{txfonts}
\usepackage{epsfig,multicol}
\usepackage[dvipdfm]{pict2e}

\theoremstyle{plain} \numberwithin{equation}{section}
\newtheorem{theo}{Theorem}[section]
\newtheorem{theo*}{Theorem}[section]
\newtheorem{coro}[theo]{Corollary}
\newtheorem{prop}[theo]{Proposition}
\newtheorem{prop*}[theo*]{Proposition}
\newtheorem{lemm}[theo]{Lemma}
\newtheorem{lemm*}[theo*]{Lemma}
\newtheorem{defi}[theo]{Definition}

\newtheorem{exam}[theo]{Example}
\newtheorem{claim}[theo]{Claim}

\thanks{The author was supported by JSPS Research Fellowships for Young Scientists.}
 \begin{document}
\title{Counting generalized Dyck paths}
\author[Y. Fukukawa]{Yukiko Fukukawa} 
\address{Department of Mathematics, Osaka City University, Sumiyoshi-ku, Osaka 558-8585, Japan.}
\email{yukiko.fukukawa@gmail.com} 

\date{\today}
\maketitle
\begin{abstract}
The Catalan number has a lot of interpretations and one of them is the number of Dyck paths. A Dyck path is a lattice path from $(0,0)$ to $(n,n)$ which is below the diagonal line $y=x$. One way to generalize the definition of Dyck path is to change the end point of Dyck path, i.e. we define (generalized) Dyck path to be a lattice path from $(0,0)$ to $(m,n) \in \mathbb{N}^2$ which is below the diagonal line $y=\frac{n}{m}x$, and denote by $C(m,n)$ the number of Dyck paths from $(0,0)$ to $(m,n)$. 
In this paper, we give a formula to calculate $C(m,n)$ for arbitrary $m$ and $n$.   
\end{abstract}
\section{Introduction}
The Catalan number $C_n=\displaystyle \frac{1}{n+1}\binom{2n}{n}$ is one of the most fascinating numbers, and it is known that the Catalan number has more than 200 interpretations. (See \cite{stanley}.)
For example, the number of ways to dissect a convex
$(n + 2)$-gon into triangles, that of binary trees with $(n+1)$ leaves, and that of standard tableaux on the young diagram $(n,n)$ are $C_n$. 
Moreover, one of the most famous interpretations of the Catalan number is the number of Dyck paths from $(0,0)$ to $(n,n)$.
A sequence of lattice points in $\mathbb{Z}^2$ 
$$P=\{ (x_0, y_0), (x_1,y_1) , \cdots , (x_k,y_k) \}$$
is a lattice path if and only if $P$ satisfies the followings for any $i=1,2,\cdots, k$: 
\begin{equation*}
(x_i, y_i)=(x_{i-1}, y_{i-1}+1) \text{ or} \  (x_{i-1}+1, y_{i-1}). 
\end{equation*}
If a lattice path $P=\{ (0,0), (x_1,y_1) , \cdots , (n,n) \}$ lies in the domain $y\leq x$, $P$ is called a Dyck path. 
There are $\binom{2n}{n}$ lattice paths from $(0,0)$ to $(n,n)$, and $C_n$ of them are Dyck paths. 

It is known that the Catalan numbers satisfy the recurrence relation that 
\begin{equation}\label{zennka1}
C_0=1 \quad \text{and} \quad 
C_n=\sum_{i=0}^{n-1}C_iC_{n-1-i}. 
\end{equation}
This recurrence relation also has  many interpretations. 
Hereafter, if a lattice path $P$ from $(0,0)$ to $(m,n) \in \mathbb{N}^2$ lies in the domain $y\leq \frac{n}{m}x$, we call $P$ a Dyck path, and we denote the number of Dyck paths from $(0,0)$ to $(m,n)$ by $C(m,n)$. 
We have a natural question: how many Dyck paths from $(0,0)$ to $(m,n)$ are there for any positive integers $m$ and $n$? 
The followings are answers to this question for special values of $m$ and $n$. 
N. Fuss (\cite{fuss}) found 
\begin{equation}
C(kn,n)=\frac{1}{kn+1}\binom{(k+1)n}{n}.   \label{Cknn}  
\end{equation}
$C(kn,n)$ also has the following recurrence relation: 
\begin{equation}
C(kn,n)=\sum_{(n_1, n_2, \cdots ,n_{k+1})}\prod_{i=1}^{k+1}C(kn_i, n_i),  \label{zennka2} 
\end{equation}
where the sum is taken over all sequences of non-negative integers $(n_1, n_2, \cdots ,n_{k+1})$ such that $\sum_{i=1}^{k+1}n_i=n-1$.
$C(kn,n)$ also appears in various counting problems, like the Catalan number. 
For instance, the number of ways to dissect a convex
$(kn + 2)$-gon into $(k + 2)$-gons is $C(kn,n)$. Actually, N. Fuss gave the formula ~\eqref{Cknn} of $C(kn,n)$ by counting the number of ways to dissect a convex
$(kn + 2)$-gon into $(k + 2)$-gons in ~\cite{fuss}.
P. Duchon \cite{PD} also gave a formula counting the number of Dyck path from $(0,0)$ to $(2\ell,3\ell)$, namely 
\begin{equation}\label{eqduchon}
C(2\ell,3\ell)=\sum_{i=0}^5 \frac{1}{5\ell+i+1}\binom{5\ell+1}{n-i} \binom{5\ell+2i}{i}.
\end{equation}

In this paper, we count $C(m,n)$ for any positive integers $m$ and $n $.
Let $A_{(m,n)}=\frac{1}{m+n}\binom{n+m}{n}$. For any $m$ and $n$, $C(m,n)$ is given by the following theorem.
\begin{theo}\footnote{When the author almost finished writing this paper, she found a paper ~\cite{mtl} which proves Theorem~\ref{intoromain}. But our proof is different from that in ~\cite{mtl}.} \label{intoromain}
Let $d={\rm  gcd}(m,n)$, then
\begin{equation}\label{introeqmain}
C{(m,n)}=\sum_{a }\prod _{i=1}^d \left( \frac{1}{a_i!}A_{(\frac{i}{d}m,\frac{i}{d}n)}^{a_i} \right) ,
\end{equation}
where the sum $\sum_a$ is taken over all sequences of non-negative integers $a=(a_1, a_2, \cdots )$ such that $\sum_{i=1}^{\infty }ia_i =d$.
When ${\rm gcd}(m,n)=1$, ~\eqref{introeqmain} reduces to the following:
\begin{equation}\label{introeqmain2}
C(m,n)=\frac{1}{m+n}\binom{m+n}{n}.
\end{equation}
\end{theo}
In fact, we prove~\eqref{introeqmain2} first and then ~\eqref{introeqmain}. 
Actually, a sequence of non-negative integers $a=(a_1, a_2, \cdots )$ in \eqref{introeqmain} characterizes the \lq \lq form" of a Dyck path, and   
it is interesting that $C(m,n)$ is given by using these sequences. 
We will see this in the last section.  
The description of $C(m,n)$ in 
Theorem~\ref{intoromain} is completely different from that of $C(kn,n)$ in \eqref{Cknn} and that of $C(2\ell , 3\ell )$ in \eqref{eqduchon}, 
and we could not deduce \eqref{Cknn} and \eqref{eqduchon} from \eqref{introeqmain} by direct computation. (However we will see that \eqref{Cknn} is a corollary of Theorem~\ref{introeqmain}.) 

The paper is organized as follows. In Section~\ref{sec2} we will treat the case ${\rm gcd}(m,n)=1$ and prove \eqref{introeqmain2}. 
In Section 3 we state a recurrence relation generalizing \eqref{zennka1} and see that Theorem~\ref{intoromain} follows from the recurrence relation. The recurrence relation follows from three lemmas and we prove them in Section 4. 
\\  \  \\
{\bf Acknowledgment.} \quad The author thanks Professor Mikiya Masuda  for helpful conversations and advice, and also thanks Professor Akihiro Munemasa for some useful 
information on the reference.  
\section{The description of $C(m,n)$ when ${\rm gcd}(m,n)=1$.}\label{sec2}
We begin with some notations about a lattice  path.  
We can regard any lattice path $P$ from $(0,0)$ to $(m,n)$ as a sequence of $m$ $x$'s and $n$ $y$'s.
For example, the lattice path from $(0,0)$ to $(5,3)$ in Figure~\ref{ehhennzu1} is $xyxxyxyx$. 
\begin{figure}[h]
\centering{
\begin{minipage}{0.3\hsize}
\includegraphics[width=30mm]{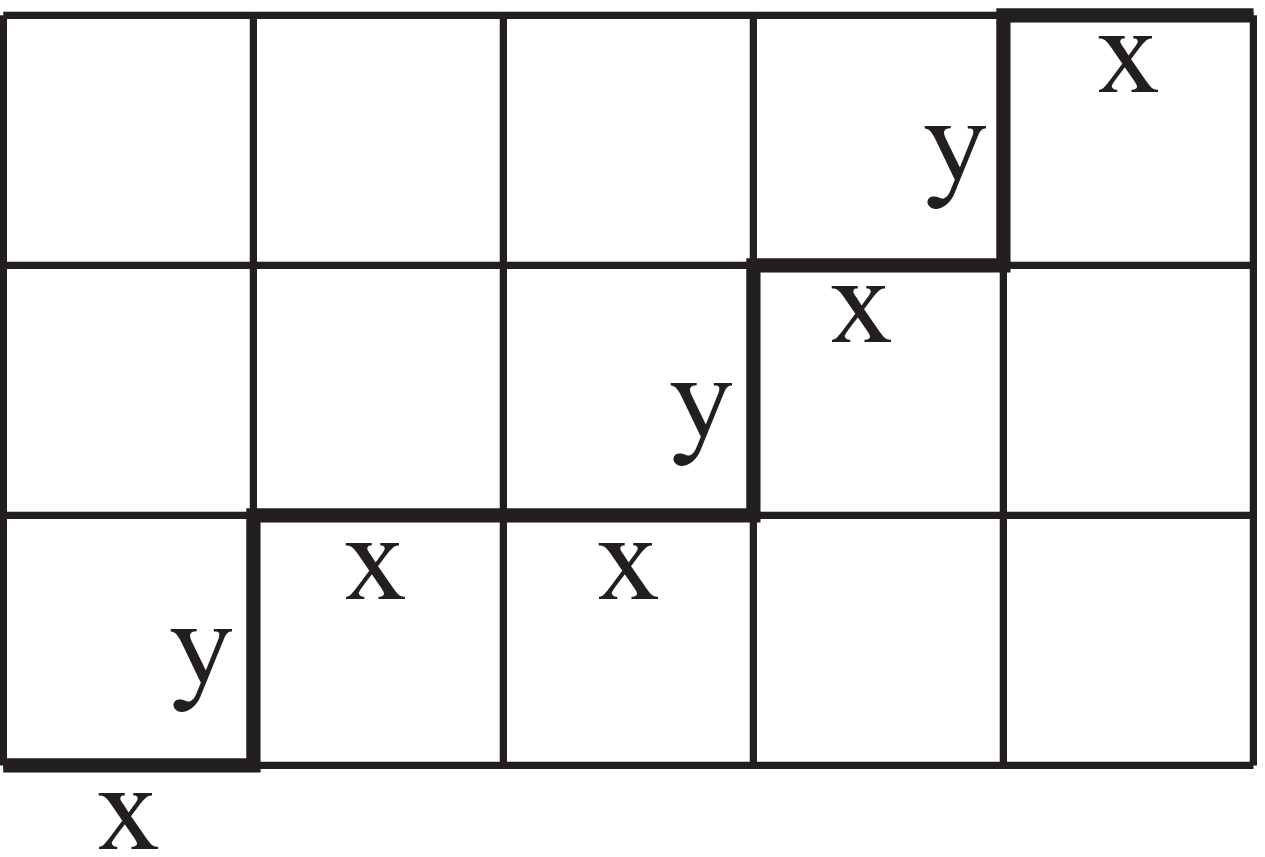}\caption{}\label{ehhennzu1}
\end{minipage}} 
\end{figure}
Hereafter we will treat a lattice path from $(0,0)$ to $(m,n)$ as a sequence of $m$ $x$'s and $n$ $y$'s. 
The number of lattice paths from $(0,0)$ to $(m,n)$ is $\binom{m+n}{n}$. 
\begin{defi}
Two lattice paths $P=u_1 u_2 \cdots u_{m+n}$ and $Q=v_1 v_2 \cdots v_{m+n}$ from $(0,0)$ to $(m,n)$ are equivalent if and only if there is some $1\leq i\leq m+n$ such that $u_{i+1}\cdots u_{m+n}u_1\cdots u_{i}=v_1 v_2 \cdots v_{m+n}$, and we denote the equivalence class of $P$ as $[P]$. 
\end{defi}
 For any lattice path $P=u_1 u_2 \cdots u_{m+n}$ its equivalence class is given by 
$$[P]=\{ P_s := u_{s+1} u_{s+2} \cdots u_{m+n} u_1 u_2 \cdots u_s \mid s= 1,2, \cdots ,m+n \}.$$ 
For example, 
when $P=xyxxy$, the elements in $[P]$ are the following five lattice paths: 
$$P_1= yxxyx, \  \  P_2=xxyxy , \  \  P_3= xyxyx, \  \  P_4 = yxyxx \  \  P=P_5=xyxxy.$$
We define {\bf the period of $P$}, denoted by ${\rm{per}}(P)$, to be the smallest number $r$ $(1\leq r \leq m+n)$ such that $P=P_r$. 
\begin{lemm}\label{ehhenn1}
For a lattice path $P$ from $(0,0)$ to $(m,n)$, $$\sharp [P]= {\rm per}(P)= (m+n)/q \  ,$$
where $q$ is a divisor of ${\rm gcd}(m,n)$. 
In particular $\sharp [P]=m+n$ if ${\rm gcd}(m,n)=1$. 
\end{lemm}
\begin{proof}
Since the lemma clearly holds if ${\rm per}(P)=m+n$,  
we assume that ${\rm per}(P)<m+n$, and let 
$m+n={\rm per}(P)q+r$ ($0 < q, \ 0 \leq r < {\rm per}(P)$). 

Assume that $r$ is zero. Then  
 $P$ is a sequence arranged $u_1 u_2 \cdots  u_{{\rm per}(P)}$ $q$ times and ${\rm per}(P)=(m+n)/q$. 
If the number of $x$ (resp. $y$) in $\{ u_1, u_2, \cdots , u_{{\rm per}(P)} \}$ is $b$ (resp. ${\rm per}(P)-b$), then we have 
$m=bq, \  n=({\rm per}(P)-b)q$. 
Thus, $q$ is a common divisor of $m$ and $n$. 

Assume that $r$ is not zero. 
Let $\bar{P}$ be a sequence arranged $P$ infinitely many times and we treat its indexes as consecutive numbers, namely $$ \bar{P} = u_{1} \cdots u_{m+n} u_{m+n+1} \cdots u_{2(m+n)} u_{2(m+n)+1} \cdots u_{3(m+n)}\cdots$$  
where  
$u_i=u_{(m+n)b +i}$. 
There are positive integers $a$ and $b$ which satisfy ${\rm per}(P)a-(m+n)b=d$, where $d={\rm gcd}({\rm per}(P), m+n)$. 
Therefore, we get the following equation: 
$$u_i=u_{i+{\rm per}(P)a}=u_{i+d+(m+n)b}=u_{i+d},$$
and this equation means $P=P_d.$
Since $r$ is not zero, $d<{\rm per}(P)$. So it is a contradiction to the minimality of ${\rm per}(P)$. 
\end{proof}
\begin{lemm}\label{ehhenn2}
For any lattice path $P$ from $(0,0)$ to $(m,n)$, $[P]$ has at least one Dyck path, and if ${\rm gcd}(m,n)=1$, $[P]$ has a unique Dyck path. 
\end{lemm}
\begin{proof}
We may assume that $m\geq n$, because $C(m,n)=C(n,m)$. 
We define a function $r$ for any pair of positive numbers $s$ and $t$, and any lattice path $Q$:  
$$r(s,t,Q)=t(\text{the number of $x$ in $Q$})-s(\text{the number of $y$ in $Q$})$$
Let $sub_i(P)$ be a subsequence of $P=u_1 \cdots u_{m+n}$ given by $u_1 u_2 \cdots u_{j_i}$ $(i=1, 2, \cdots ,n)$, where $u_{j_i}$ is the $i^{th}$ $y$ in $P$ from the left. 
Since the Dyck path is a lattice path which is below the diagonal line $y=\frac{n}{m}x$, 
the definition of Dyck path can be described in terms of the function $r$ as follows:  
$$r(m,n, sub_i(P))\geq 0 \quad \quad \text{for any } 1\leq i \leq n.$$

Suppose that a lattice path $P$ from $(0,0)$ to $(m,n)$ is not a Dyck path and 
let $k$ be the positive integer such that the function $r(m,n,\cdot )$ takes the minimum value on $sub_{k}(P)$Cthen $P_k$ is a Dyck path. 
To prove this, we should confirm that $r(m,n, sub_{i}(P_k)) \geq 0 $ $(\text{for any}  \  i)$, but we can see this easily.  
See Figure~\ref{ehhennzu2}. This is the figure of a part of lattice path $\bar{P}$ and the line with the slope $\frac{n}{m}$ which is over $\bar{P}$ and touches $\bar{P}$ at only lattice points. 
For any lattice path $P$, there is a unique such line. To observe $P_s$ for any $P$ is same as to observe some subsequence with length $m+n$ of $\bar{P}$. 
\begin{figure}[h]
\centering{
\includegraphics[width=60mm]{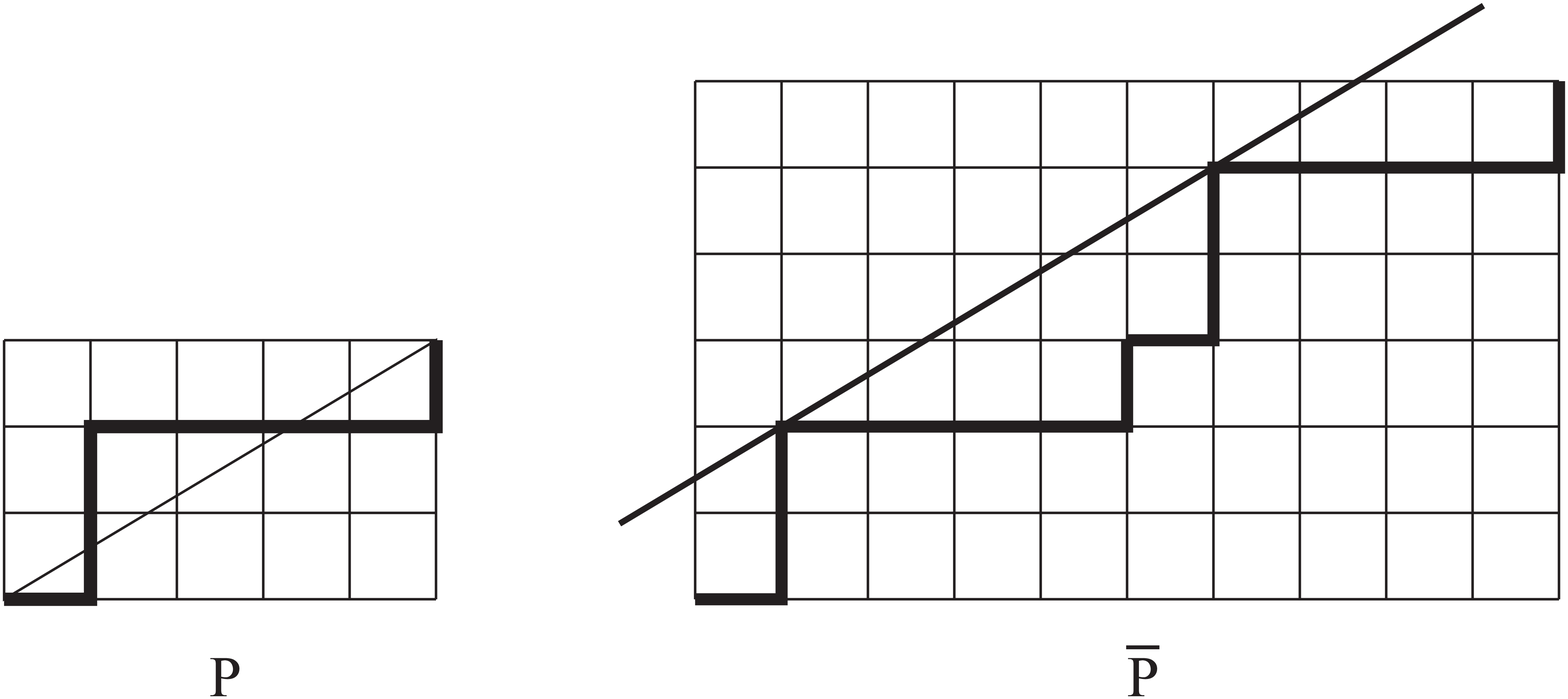}\caption{$(m,n)=(5,3)$}\label{ehhennzu2}
} 
\end{figure}
Choose two common points such that the difference of $x$-coordinates is $m$, and regard those two points as $(0,0)$ and $(m,n)$ from the left, that lattice path is a Dyck path. 
Therefore any lattice path  $P$ has at least one Dyck path in their equivalence class. 
$k$ is  one of $x$-coordinates of common points of $\bar{P}$ and the line. 
When ${\rm gcd}(m,n)=1$, the difference of $x$-coordinates of two adjacent lattice points on that line is $m$, so the Dyck path in $[P]$ is unique, as desired.    
\end{proof}
\begin{theo}\label{ehhenn}
If ${\rm gcd}(m,n)=1$, then 
$$C(m,n)=\frac{1}{m+n}\binom{m+n}{n}.$$
\end{theo}
\begin{proof}
We can choose some lattice paths $P^1, P^2, \cdots ,P^t$ from $(0,0)$ to $(m,n)$ such that the set of all lattice paths from $(0,0)$ to $(m,n)$ can be written as the following: 
\begin{equation}\label{ehhenneq1}
\{ \text{All lattice paths from $(0,0)$ to $(m,n)$} \}=[P^1] \sqcup [P^2] \sqcup \cdots \sqcup [P^t]
\end{equation}
Lemma ~\ref{ehhenn2} says that each $[P^i]$ has a unique Dyck path if ${{{\rm gcd}}(m,n)}=1$, so $t=C(m,n)$. 
Comparing the number of elements in both side of ~\eqref{ehhenneq1}, 
\begin{eqnarray*}
\binom{m+n}{n} &=& \mid [P^1] \sqcup [P^2] \sqcup \cdots \sqcup [P^t] \mid \\
               &=& (m+n)t  \quad \quad (\because \text{Lemma~\ref{ehhenn1}})\\
               &=& (m+n)C(m,n).
\end{eqnarray*}
Therefore we have $$C(m,n)=\displaystyle{\frac{1}{m+n}} \binom{m+n}{n},$$ and Theorem~\ref{ehhenn} is proven. 
\end{proof}
It is easy to show that ~\eqref{Cknn} is given as a corollary of Theorem~\ref{ehhenn}. 
\begin{coro}\label{coro1}
$$C(kn,n)=\frac{1}{kn+1}\binom{(k+1)n}{n}.$$
\end{coro}
\begin{proof}
Note that $C(n,kn)=C(n,kn+1)$ holds. In fact, since ${\rm gcd}(kn+1,n)=1$, the lattice points in the domain $\{(x,y) \mid y \leq \frac{kn+1}{n}x, \  y \geq  kx, \   0 \leq x \leq n \}$ are on the line 
$y= kx$ or $(n , kn+1)$. Namely, all Dyck paths from $(0,0)$ to $(n,kn+1)$ are made by connecting the lattice path from $(n,kn)$ to $(n,kn+1)$ with a lattice path from $(0,0)$ to $(n,kn)$. 
Therefore,  we have
\begin{eqnarray*}
C(kn,n)&=&C(n,kn)=C(n,kn+1)=\frac{1}{(k+1)n+1}\binom{(k+1)n+1}{n} \nonumber \\
      &=&\frac{1}{kn+1}\binom{(k+1)n}{n}. 
\end{eqnarray*}
and this is \eqref{Cknn}.
\end{proof}
\section{The description of $C(m,n)$ for any positive integers $m$ and $n$.}\label{sec3}
In this section, we describe the formula of $C(m,n)$ for any positive integers $m$ and $n$. 
Let $A_{(m,n)}=\frac{1}{m+n}\binom{m+n}{n}$, then the following Proposition holds. 
\begin{prop}\label{cac}
Let $d={\rm gcd}(m,n)$, then 
\begin{equation}\label{zennka3-1}
C{(m,n)} =\sum_{i=1}^d\frac{i}{d}A_{(\frac{i}{d}m,\frac{i}{d}n)} C{(\frac{d-i}{d}m,\frac{d-i}{d}n)}.
\end{equation}
\end{prop}
The proof of the proposition will be given later.

This recurrence relation \eqref{zennka3-1} is a generalization of the recurrence relation ~\eqref{zennka1}. 
In fact, when $m=n$, $d=n$ and \eqref{zennka3-1} reduces to   
$$C_n=\sum_{i=1}^n\frac{i}{n}A_{(i,i)} C_{n-i}. $$
Here a  part of the right hand side, 
$\frac{i}{n}A_{(i,i)} C_{n-i}+\frac{n-i+1}{n}A_{(n-i+1,n-i+1)} C_{i-1}$, is equal to $2C_{n-i}C_{i-1}$, since  
\begin{eqnarray*}
& &\frac{i}{n}A_{(i,i)} C_{n-i}+\frac{n-i+1}{n}A_{(n-i+1,n-i+1)} C_{i-1}\\
&=&\frac{i}{n}\frac{1}{2i}\binom{2i}{i} C_{n-i}+\frac{n-i+1}{n}\frac{1}{2(n-i+1)}\binom{2(n-i+1)}{n-i+1} C_{i-1}\\
&=&\frac{1}{2n}\frac{1}{i}\binom{2(i-1)}{i-1} \frac{2i(2i-1)}{i} C_{n-i}+\frac{1}{2n}\frac{1}{n-i+1}\binom{2(n-i)}{n-i} \frac{2(n-i+1)(2n-2i+1)}{n-i+1}C_{i-1}\\
&=&\left( \frac{2i-1}{n}+\frac{2n-2i+1}{n} \right) C_{n-i} C_{i-1}\\
&=&2C_{n-i}C_{i-1}. 
\end{eqnarray*}
So, \eqref{zennka3-1} is a generalization of ~\eqref{zennka1}.
\begin{defi}
For a sequence of non-negative integers $a=(a_1 , a_2 , \cdots )$, we define $\| a \| $ by 
$$\| a \| = \sum_{i=1}^\infty  ia_i .$$
\end{defi}
The recurrence relation ~\eqref{zennka3-1} leads to the main theorem. 
\begin{theo}\label{main}
Let $d={\rm gcd}(m,n)$, then 
\begin{equation}\label{eqmain}
C{(m,n)}=\sum_{a ; \parallel  a \parallel  =d}\prod _{i=1}^d \left( \frac{1}{a_i!}A_{(\frac{i}{d}m,\frac{i}{d}n)}^{a_i} \right) .
\end{equation}
\end{theo}
\begin{exam}
Let $m=n$, then we have 
$$C(n,n)=\sum_{a; \| a \|=n} \prod_{i=1}^n \left( \frac{1}{a_i!} A_{(i,i)}^{a_i} \right) $$
by ~\eqref{eqmain}. 
For instance, in the case of $n=3$, the sequences of non-negative integers $a=(a_1, a_2, \cdots )$ with 
$\| a \| =3$ are the following three sequences. 
$$(3, 0, 0, \cdots ), \quad (1,1,0,0, \cdots ), \quad (0,0,1,0,0, \cdots ).$$
Thus, 
\begin{eqnarray*}
C(3,3)&=& \sum_{a; \| a \| =3} \prod_{i=1}^3 \left( \frac{1}{a_i!} A_{(i,i)}^{a_i} \right) \\
      &=& \frac{1}{3!}A_{(1,1)}^3 \\
      & & \quad +\frac{1}{1!}A_{(1,1)}^1 \cdot \frac{1}{1!}A_{(2,2)}^1 +\frac{1}{1!}A_{(3,3)}^1 \\
      &=& \frac{1}{6} + \frac{3}{2} +\frac{20}{6} = 5.
\end{eqnarray*}
As this example shows, each factor $\prod _{i=1}^d \left( \frac{1}{a_i!}A_{(\frac{i}{d}m,\frac{i}{d}n)}^{a_i} \right)$ in the sum is not necessarily an integer. 
\end{exam}
For simplicity, we rewrite Proposition~\ref{cac} and Theorem~\ref{main}.   
\begin{prop}[Proposition ~\ref{cac}]\label{cac'}
Let $p$ and $q$ be  two positive integers with ${\rm gcd}(p,q)=1$. We denote the number of Dyck paths from $(0,0)$ to $(dp,dq)$ (namely, $C(dp,dq)$) by $\widetilde{C}_d$. Likewise, we abbreviate $A_{(dp,dq)}$ as $A_d$. 
Then, we have
\begin{equation*}
\widetilde{C}_d=\sum_{i=1}^d \frac{i}{d}A_i\widetilde{C}_{d-i}.
\end{equation*}
\end{prop}
\begin{theo}[Theorem ~\ref{main}]\label{main'}
Under the same assumption of Propsition~\ref{cac'}, we have
$$\widetilde{C}_d=\sum_{a; \| a \| =d} \prod_{i=1}^d\left( \frac{1}{a_i!}A^{a_i}_i \right).$$
\end{theo}
 
In the rest of this section, we show that Theorem~\ref{main'} follows from Proposition~\ref{cac'}, state three lemmas and prove Proposition~\ref{cac'} using them. 
For that we need some notations. 
\begin{defi}
For a sequence of non-negative integers $a=(a_1 , a_2 , \cdots , )$, $\mid a \mid$, $\ell (a)$ and $h(a)$ are defined by 
\begin{equation*}
\mid a \mid = \sum_{i=1}^\infty  a_i ,  \quad  \ell(a) = \sharp  \{i \mid a_i \not = 0 \} , \quad \text{and} \quad 
h(a)= {\displaystyle \frac{ \mid a \mid ! }{\prod_{i=1}^\infty  a_i ! }} 
\end{equation*}
respectively. 
\end{defi}
If $\| a \| =d < \infty $, $h(a)= {\displaystyle \frac{ \mid a \mid ! }{\prod_{i=1}^d a_i ! }}$. 
Hereafter we always assume $a$ satisfies $\ell(a)<\infty $. 
\begin{defi}
For two sequences of non-negative integers $a=(a_1 , a_2 , \cdots )$ and $c=(c_1 , c_2 , \cdots )$, we define 
$$c\geq a \Leftrightarrow c_i \geq a_i \quad \text{for any } i=1, 2 ,\cdots  $$
$$ c> a \Leftrightarrow c \geq a \quad \text{and} \quad   c_i \not = a_i \text{ \  for some \  } i=1, 2 ,\cdots .$$
Moreover, for $0\leq  j < \mid c \mid $, let $B_c^j$ be 
$$B_c^j:=\{ a \mid a \leq c ,  \mid a \mid =\mid c \mid -j \} .$$
\end{defi}
\begin{defi}
Suppose that $p$ and $q$ are a pair of positive integers with ${\rm gcd}(p,q)=1$, and $d$ is any positive integer. 
Let $D(dp,dq)$ be the number of Dyck paths from  $(0,0)$ to $(dp,dq)$ which is strictly below the diagonal $y=\frac{p}{q}x$ except at $(0,0)$ and $(dp,dq)$, 
and $D(0,0)=1$. 
For any sequence of non-negative integers $a$ with $\| a\| < \infty $, we set $$D^a _{p,q}:=\prod_{i=1}^{\infty}D{(ip,iq)}^{a_i}. $$
If $p$ and $q$ are clear from the context, we abbreviate $D^a_{p,q}$ as $D^a$ and abbreviate $A_{(dp,dq)}$ as $A_d$ as before.  
\end{defi}
\begin{lemm}\label{chaD}
\begin{equation}\label{chaDeq}
\widetilde{C}_d=\sum_{a ; \parallel  a \parallel  =d}h(a)D^a .
\end{equation}
\end{lemm}
\begin{lemm}\label{ahaD}
\begin{equation}\label{ahaDeq}
A_d=\sum_{a ; \parallel a \parallel =d}\frac{1}{\mid a \mid}h(a)D^a .
\end{equation}
\end{lemm}
\begin{lemm}\label{hh}
For any sequence of non-negative integers $c=(c_1, c_2 , \cdots ) $ and any j with $0 \leq j \leq \mid c \mid -1$, the following holds. 
$$\sum_{a \in B_c^j} \frac{\parallel a \parallel }{\mid a \mid }h(a)h(c-a)=\frac{\parallel c\parallel }{\mid c \mid }h(c).$$
\end{lemm}
We will give the proofs of these lemmas in the next section.
Here we assume that Lemma~\ref{chaD}, Lemma~\ref{ahaD}, and Lemma~\ref{hh} are correct and give the proof of Proposition~\ref{cac'}. 
 \begin{proof}[Proof of Proposition~\ref{cac'}]
Now we fix the pair of positive integers $p$ and $q$ with ${\rm gcd}(p,q)=1$. 
Substituting \eqref{chaDeq} and \eqref{ahaDeq} for the right hand side of Proposition~\ref{cac'}, we have 
\begin{eqnarray}\label{eqdc}
\sum_{i=1}^d\frac{i}{d}A_i \widetilde{C}_{d-i} &=&\frac{1}{d} \sum_{i=1}^d \left( \sum_{\parallel a \parallel =i} \frac{i}{\mid a \mid}h(a)D^{a} \right)\left( \sum_{\parallel b \parallel =d-i} h(b) D^b  \right) \nonumber \\
&=& \frac{1}{d} \sum_{i=1}^d  \left(\sum_{\parallel a \parallel =i} \sum_{\parallel b \parallel =d-i} \frac{i}{\mid a \mid}h(a)h(b) D^{a+b}  \right) \nonumber \\
&=& \frac{1}{d}  \sum_{\parallel c \parallel =d} \left( \sum_{a \leq c } \frac{\parallel a \parallel }{\mid a \mid}h(a)h(c-a)\right) D^{c}, 
\end{eqnarray}
where the last equality in \eqref{eqdc} is given by substituting $c=a+b$. 
Calculating the factor in the right hand side of \eqref{eqdc}: 
\begin{eqnarray*}
\sum_{a \leq c } \frac{\parallel a \parallel}{\mid a \mid}h(a)h(c-a)&=&\sum_{j=0}^{\mid c \mid -1} \sum_{a \in B_c^j}\frac{\parallel a \parallel}{\mid a \mid}h(a)h(c-a)\\
&=&\sum_{j=0}^{\mid c \mid -1} \frac{\parallel c\parallel }{\mid c \mid }h(c)\\
&=&{\parallel c\parallel }h(c).
\end{eqnarray*}
The second equality above follows from Lemma~\ref{hh}, thus we have  
\begin{eqnarray*} 
\sum_{i=1}^d\frac{i}{d}A_i \widetilde{C}_{d-i}&=&\frac{1}{d}  \sum_{\parallel c \parallel =d} \left( \sum_{a \leq c } \frac{\parallel a \parallel }{\mid a \mid}h(a)h(c-a)\right) D^{c}\\
                                              &=& \sum_{c  ; \parallel c \parallel =d} h(c) D^{c}  \\
              &=& \widetilde{C}_d
\end{eqnarray*}
by Lemma~\ref{chaD}. 
Therefore Proposition~\ref{cac'} is proven. 
\end{proof}
Theorem~\ref{main'} follows from Proposition~\ref{cac'}. 
\begin{proof}[Proof of Theorem~\ref{main'}]
Fix the pair of positive integers $p$ and $q$ with ${\rm gcd}(p,q)=1$. We prove Theorem~\ref{main'} by induction on $d$. When $d=1$, since $(m,n)=(1\cdot p ,1\cdot q)$, $C(m,n)=\widetilde{C}_1=A_1$ follows Theorem~\ref{ehhenn}Cthus  Theorem~\ref{main'} holds. Assume that Theorem~\ref{main'} holds for less than or equal to $d-1$. 
Then we have 
\begin{eqnarray}
\widetilde{C}_d&=&\sum_{i=1}^d \frac{i}{d} A_i\widetilde{C}_{d-i} \nonumber \\
&=& \sum_{i=1}^d \frac{i}{d} A_{i} \left( \sum_{a; \parallel  a \parallel  =d-i}\prod_{j=1}^{d-i}\frac{1}{a_j!}A_{j}^{a_j} \right)  \label{eq1}
\end{eqnarray}
by Propsition~\ref{cac'} and the induction assumption.  

\begin{claim}
The following equation holds for $d$ variables $x_1, x_2 , \cdots , x_d$:  
\begin{equation}\label{coef}
\sum_{i=1}^d \frac{i}{d} x_i \left( \sum_{\| a \| =d-i} \prod_{j =1}^{d-i} \frac{1}{a_j !} x_j^{a_j} \right) = \sum_{\| a \| =d} \prod_{i=1}^d \frac{1}{a_i !} x_i^{a_i}. 
\end{equation}
\end{claim}
\begin{proof}[Proof of the claim]
For any sequence of non- negative integers $$b=(b_1 ,b_2 , \cdots ,b_d , 0 ,0 , \cdots),$$ 
we shall observe the coefficients of $x_{1}^{b_{1}}x_{2}^{b_{2}}\cdots x_{d}^{b_{d}}$ in  the left hand side of ~\eqref{coef}. 
\\ \ 
(1) If $\parallel b \parallel = d$, the term which contains $x_{1}^{b_{1}}x_{2}^{b_{2}}\cdots x_{d}^{b_{d}}$ in the left hand side of \eqref{coef} is 
\begin{equation}\label{eq2}
\sum_{i=1}^d \frac{i}{d} x_i \left( \frac{1}{(b_i-1)!}x_i^{b_i-1} \prod_{j\not= i} \frac{1} {b_j!} x_j^{b_j} \right), 
\end{equation}
where we understand $\frac{1}{(b_i-1)!}x_i^{b_i-1}=0$ if $b_i=0$,  
and \eqref{eq2} is equal to  
\begin{equation*}
\sum_{i=1}^d\frac{ib_i }{d}\left( \prod_{j=1}^d \frac{1}{b_j!}x_j^{b_j} \right) = \left( \frac{1 }{d}\sum_{i=1}^d ib_i \right) \left( \prod_{j=1}^d \frac{1}{b_j!}x_j^{b_j} \right)=\prod_{j=1}^d \frac{1}{b_j!}x_j^{b_j}. 
\end{equation*}
 \  
(2) We shall show that any monomial in the left hand side of \eqref{coef} is of the form $x_{1}^{b_{1}}x_{2}^{b_{2}}\cdots x_{d}^{b_{d}}$ with $\| b \| =d$. 
Any monomial in the left hand side of \eqref{coef} is of the form $x_i \prod_{j =1}^{d-i} x_j^{a_j} $ with $a$ such that $\| a \|= d-i$. 
Set $b=(a_1 , \cdots ,a_{i-1}, a_{i}+1, a_{i +1}, \cdots )$. Then $\parallel b \parallel = i+ \| a \| =i + d-i=d$ 
and $x_i \prod_{j =1}^{d-i}x_j^{a_j}=x_{1}^{b_{1}}x_{2}^{b_{2}}\cdots x_{d}^{b_{d}}$ with $\parallel b \parallel = d$. 
\end{proof}
By substituting $x_i= A_i$ for the left hand side of \eqref{coef}, we have 
$$\sum_{i=1}^d \frac{i}{d} A_i \left( \sum_{a; \parallel  a \parallel  =d-i}\prod_{j=1}^{d-i}\frac{1}{a_j!}A_j^{a_j} \right) = \sum_{b ; \parallel b \parallel =d}\prod_{j=1}^d\frac{1}{b_j!}A_j^{b_j}.$$
Therefore, by \eqref{eq1}, we have
$$\widetilde{C}_d=\sum_{i=1}^d \frac{i}{d} A_{i} \left( \sum_{a; \parallel  a \parallel  =d-i}\prod_{j=1}^{d-i}\frac{1}{a_j!}A_{j}^{a_j} \right) = \sum_{b ; \parallel b \parallel =d}\prod_{j=1}^d\frac{1}{b_j!}A_j^{b_j}, $$
and Theorem~\ref{main'} follows.  
\end{proof}
\section{Proofs of the Lemmas}\label{sec4}
We fix the pair of positive numbers $p$ and $q$ with ${\rm gcd}(p,q)=1$ as before. Let $$m=dp, \quad n=dq.$$
To start with, we give the definitions of {\bf shape} $e=(e_1, e_2, \cdots )$  and {\bf type} $a=(a_1,a_2, \cdots )$ {\bf of the Dyck path } from $(0,0)$ to $(m,n)$. 
Any Dyck path $P$ touches the diagonal $y=\frac{n}{m}x$ at least one point except at $(0,0)$, and coordinates of intersection of $P$ and the diagonal can be described as $(\frac{k}{d}m, \frac{k}{d}n)=(kp,kq)$ for some $k \in \mathbb{Z}_{>0}$ because ${\rm gcd}(m,n)=d$. 
Let all intersection points of $P$ and the diagonal be $(0,0)$, $ (k_1p,k_1q)$, $(k_2p,k_2q)$, $\cdots$, $(k_sp,k_sq)$ from the left. (Namely, $0< k_1 <k_2 <\cdots <k_s=d$.) 
Then, the shape $e=(e_i)_{i \in \mathbb{N}}$ of a Dyck path $P$ is defined by $e_i=k_i-k_{i-1}$ for any non-negative integer $i$, where $k_0=0$ and $k_t=0 \  (t> s) $. 
Furthermore, the type $a=(a_i)_{i\in \mathbb{N}}$ of a Dyck path $P$ is defined by $a_i=\sharp \{e_j \mid e_j=i\}$ for any $i \geq 1$. We denote the type of $P$ by ${\rm{type}}(P)$.  
\begin{proof}[Proof of Lemma~\ref{chaD}]
Suppose that $P$ is a Dyck path form $(0,0)$ to $(m,n)$ of shape $e$ and type $a$. Then  
\begin{eqnarray*}
\| a \| &=& \sum_{i=1}^\infty ia_i= \sum_{i=1}^\infty i\sharp \{ e_j \mid e_j=i \} \\
        &=& \sum_{i=1}^\infty i\sharp \{ k_j \mid k_j-k_{j-1}=i \} \\
        &=& \sum_{i=1}^\infty (k_i- k_{i-1})= -k_0+k_s \\
        &=&  d .
\end{eqnarray*}
Conversely, for any sequence of non-negative integers $a$ with $\| a \| =d$, it is clear that there exists some Dyck path of type $a$ from $(0,0)$ to $(m,n)$. 
The number of Dyck paths of shape $e$ is $\prod_{i=1}^\infty D_{(e_ip,e_iq)}$. If the shapes of two Dyck paths coincide, their types also coincide; so the number of Dyck paths  of type $a$ is 
$$\sum_{e} \prod_{i=1}^\infty D_{(e_ip,e_iq)} = \sum_{e} \prod_{i=1}^\infty D_{(ip,iq)}^{a_i}=h(a)D^a, $$
where the sum $\sum_{e}$ is taken over all $e$ with $\sharp \{e_j \mid e_j=i\} =a_i$ for any $i$. 
This proves Lemma~\ref{chaD}. 
\end{proof}
\begin{proof}[Proof of Lemma~\ref{ahaD}]
If $[P]$ has more than one Dyck path, then types of these Dyck paths coincide. 
Let $P$ be a Dyck path of type $a$ and period $r(\not =m+n)$. Lemma~\ref{ehhenn1} says that $r$ is a divisor of $m+n$, and this means that 
$$\exists a'=(a'_1, a'_2, \cdots ) \quad s.t. \quad  \ a_i=\frac{m+n}{r} a'_i \quad \forall i \in \mathbb{N}.$$
Namely, any $a_i$ is divisible by $(m+n)/{r}$. 
Let $cd(a) $ be the set of all $r$ such that $(m+n)/{r}$ divides all $a_i$, in other words, $r$ which can be the period of some Dyck path with type $a$.  
For any Dyck path $P$ with type $a$ and period $r$, the number of lattice paths in $[P]$ is $r$ and that of Dyck paths in $[P]$ is $\mid a' \mid =r \mid a \mid /{( m+n)}$ by Lemma~\ref{ehhenn1}. 
Let $E(a,r)$ be the number of Dyck paths with type $a$ and period $r$. 
Counting the number of all lattice paths from $(0,0)$ to $(m,n)$, we have 
\begin{eqnarray*}
(m+n)A_{(m,n)}&=& \sum_{a  ;   \| a \| =d }\sum_{r ;  r\in cd(a)} \sum_{P; {\rm{type}}(P)=a, {\rm{per}}(P)=r} \frac{\sharp [P]}{(\text{Number of Dyck paths in $[P]$})}\\
              &=& \sum_{a  ;   \| a \| =d }\sum_{r ;  r\in cd(a)}  E(a,r)\frac{r}{\frac{r \mid a \mid }{ m+n}} \\
              &=& \sum_{a   ;   \| a \| =d} \sum_{r ;   r\in cd(a)} \frac{m+n}{\mid a \mid } E(a,r)\\
              &=& \sum_{a ; \| a \| =d }\frac{m+n}{\mid a \mid } \left( \sum_{r; r\in cd(a)} E(a,r)\right).
\end{eqnarray*}
We know $h(a)D^a=\sum_{r; r\in cd(a)}E(a,r)$, thus,
\begin{equation*}
(m+n)A_{(m,n)}=\sum_{a; \| a \| =d }\frac{m+n}{\mid a \mid } h(a)D^a.
\end{equation*}
Lemma~\ref{ahaD} is proved. 
\end{proof}
\begin{proof}[Proof of Lemma~\ref{hh}]
We begin with the following claim.  
\begin{claim}\label{lemha}
$$\sum_{a \in B_c^j} \frac{\parallel a \parallel }{\mid a \mid }h(a)h(c-a)=\sum_{c' \in B_c^1}\sum_{a \in B_{c'}^{j-1}}\frac{\parallel a \parallel }{\mid a \mid }h(a)h(c'-a)$$
\end{claim}
\begin{proof}[Proof of the claim]
Recall $\ell(c)=\# \{ i \mid c_i \not = 0 \}$. 
For any sequence of non-negative integers $c$, 
the elements in $B_c^1$ are the following $\ell(c)$ sequences:  
$$1 \leq  \forall t \leq \ell(c), \quad c^t:=(0, \cdots ,0 ,c_{s_1} ,\cdots , c_{s_t}-1, 0, \cdots ,c_{s_{\ell(c)}}, 0 , \cdots )$$
where $c_{s_t}$ is the $t^{th}$ nonzero number in $c$ from the left. 
Then, we have 
\begin{eqnarray}\label{lemeq1}
\sum_{c'\in B_c^1}h(c')&=& \sum_{t=1}^{\ell (c)} h(c^t) \nonumber \\
                       &=& \sum_{t=1}^{\ell (c)} \frac{(\mid c \mid -1)!}{\prod_{i=1}^{\ell(c)}(c_{s_{i}}!)/c_{s_t}} \nonumber \\
                       &=& \frac{\mid c \mid !}{\mid c \mid } \frac{1}{\prod_{i=1}^{\ell(c)}(c_{s_{i}}!)}\sum_{t=1}^{\ell (c)} c_{s_t} \nonumber \\
                       &=& \frac{\mid c \mid !}{\prod_{i=1}^{\ell(c)}(c_{s_{i}}!)}=h(c)
\end{eqnarray} 
We note that  
$$a \in B_c^j \Longleftrightarrow a \in B_{c'}^{j-1} \text{for some \ }c' \in B_c^1.$$
Therefore we have 
\begin{eqnarray*}
\sum_{c' \in B_c^1}\sum_{a \in B_{c'}^{j-1}}\frac{\parallel a \parallel }{\mid a \mid }h(a)h(c'-a)&=&\sum_{a \in B_c^j}\frac{\| a\| }{\mid a\mid}\left( \sum_{c'\in B_c^1 s.t. c'>a}h(c'-a)\right)h(a)\\
                                                                                                  &=&\sum_{a \in B_c^j}\frac{\| a\| }{\mid a\mid}\left( \sum_{b \in B_{c-a}^1} h(b)\right)h(a)\\
                                                                                                  &=&\sum_{a \in B_c^j}\frac{\| a\| }{\mid a\mid}h(c-a)h(a).
\end{eqnarray*}
The last equation above holds by ~\eqref{lemeq1}. 
Therefore Claim~\ref{lemha} is proved.
\end{proof} 
We go back to the proof of Lemma~\ref{hh}. 
We prove by induction on $j$.  Lemma~\ref{hh} clearly holds for $j=0$. 
We have $$B_c^1=\{ c^t=(0, \cdots ,0 ,c_{s_1} ,\cdots , c_{s_t}-1, 0, \cdots ,c_{s_{\ell(c)}}, 0 , \cdots  ) \mid 1 \leq  \forall t \leq \ell(c) \},$$
 and then 
\begin{eqnarray*}
\sum_{a \in B_c^1} \frac{\parallel a \parallel }{\mid a \mid }h(a)h(c-a)&=&\sum_{t=1}^{\ell(c)}\frac{\parallel c^t \parallel }{\mid c^t \mid}h(c^t)h(c-c^t)\\
                                                                        &=&\sum_{t=1}^{\ell(c)}\frac{\parallel c \parallel -s_t }{\mid c \mid -1}\frac{(\mid c \mid -1)!}{(\prod_{i=1}^{\ell(c)}c_{s_{i}}!)/c_{s_t}}\cdot 1\\
                                                                        &=&\sum_{t=1}^{\ell(c)}\frac{\| c \| -s_t}{\mid c \mid -1}\frac{\mid c \mid !}{\mid c \mid}\frac{c_{s_t}}{\prod_{i=}^{\ell(c)}c_{s_{i}}!}\\
                                                                        &=& \left( \sum_{t=1}^{\ell(c)}c_{s_t}\| c \| - \sum_{t=1}^{\ell(c)}s_tc_{s_t} \right)\frac{1}{\mid c \mid (\mid c \mid -1)}h(c)\\
                                                                        &=&(\| c \| \mid c \mid -\| c \| )\frac{1}{\mid c \mid (\mid c \mid -1)}h(c)\\
                                                                        &=&\frac{\| c \| }{\mid c \mid}h(c).
\end{eqnarray*}
Thus, Lemma~\ref{hh} holds for $j=1$. 
Assume that $j\geq 2$ and Lemma~\ref{hh} holds for $j-1$. 
By Claim~\ref{lemha}, we have  
\begin{eqnarray*}
\sum_{a \in B_c^j} \frac{\parallel a \parallel }{\mid a \mid }h(a)h(c-a)&=&\sum_{c' \in B_c^1}\sum_{a \in B_{c'}^{j-1}}\frac{\parallel a \parallel }{\mid a \mid }h(a)h(c'-a)\\
                                                                        &=&\sum_{c' \in B_c^1}\frac{\parallel c' \parallel }{\mid c' \mid}h(c')\\
                                                                        &=&\frac{\parallel c \parallel }{\mid c \mid}h(c), 
\end{eqnarray*}
so Lemma~\ref{hh} also holds for $j$.  
\end{proof}

 \end{document}